\documentclass[3p,10pt,a4paper,twoside,fleqn,sort&compress]{elsarticle}
\usepackage{amssymb,amsmath,latexsym}
\usepackage{pst-all}
\usepackage[varg]{pxfonts}
\usepackage{mathrsfs}
\font\tenscr=rsfs10  scaled \magstep0 \font\sevenscr=rsfs7 scaled \magstep0
\font\fivescr=rsfs5 scaled \magstep0 \skewchar\tenscr='177 \skewchar\sevenscr='177
\skewchar\fivescr='177
\newfam\scrfam \textfont\scrfam=\tenscr \scriptfont\scrfam=\sevenscr
\scriptscriptfont\scrfam=\fivescr
\def\mathscr{\fam\scrfam}

\font\stenscr=rsfs5  scaled \magstep3
\newfam\sscrfam \textfont\sscrfam=\stenscr

\newtheorem{thm}{Theorem}[section]
\newtheorem{lemma}[thm]{Lemma}
\newtheorem{cor}[thm]{Corollary}
\newtheorem{rem}[thm]{Remark}
\newtheorem{example}[thm]{Example}
\newtheorem{prop}[thm]{Proposition}

\def\endproof{\qed\endtrivlist}
\expandafter\let\csname endproof*\endcsname=\endproof

\def\qedsymbol{\ifmmode\bgroup\else$\bgroup\aftergroup$\fi
  \vcenter{\hrule\hbox{\vrule height.6em\kern.6em\vrule}\hrule}\egroup}
\def\qed{\ifmmode\else\unskip\nobreak\fi\quad\qedsymbol}

\begin{document}

\journal{}

\title{\Large\bf Lattice induced threshold functions and Boolean functions}
%%%
%\tnotetext[t1]{Research supported by Ministry Education and Science, %Republic of Serbia, Grant No. 174013}
%%%
%%%%%%%%%%%%%%%%%%%%%%%%%%%%%%%%%%%%%%%%%%%%%%%%%%
\author[szeg]{Eszter K. Horv\'ath}
\ead{horeszt@math.u-szeged.hu}

\author[fsuns]{Branimir \v Se\v selja}
\ead{seselja@dmi.uns.ac.rs}

\author[fsuns]{Andreja Tepav\v cevi\'c\corref{cor}}
\ead{andreja@dmi.uns.ac.rs}

\cortext[cor]{Corresponding author.\ Tel.: +381214852862; fax: +381216350458.}

\address[szeg]{Bolyai Institute, University of Szeged,
Aradi v\'ertan\'uk tere 1, Szeged 6720, Hungary}

\address[fsuns]{University of Novi Sad, Faculty of Science, Trg D.\ Obradovi\'ca 4, 21000 Novi Sad, Serbia}
%%%%%%%%%%%%%%%%%%%%%%%%%%%%%%%%%%%%%%%%%%%%%%%%%%

\begin{abstract}\small

Lattice induced threshold function is a Boolean function determined by a particular linear combination of lattice elements. We prove that every isotone Boolean function is a lattice induced threshold function and vice versa. We also represent lattice valued up-sets on a finite Boolean lattice in the framework of cuts and lattice induced threshold functions. In terms of closure systems we present necessary and sufficient conditions for a representation of lattice valued up-sets on a finite Boolean lattice by linear combinations of elements of the co-domain lattice.

\end{abstract}

\maketitle

%{AMS subject classification: 06E30.  }
\section{Introduction}
,,Threshold functions provide a simple but fundamental model for many questions investigated in electrical engineering, artificial intelligence, game theory and many other areas." (Quotation from \cite{CH}.)
In \cite{M}, modeling neurons  and political decisions are also mentioned, as application of classical threshold functions.
A classical \textbf{threshold function} is a Boolean function $f: \{0,1\}^n\rightarrow \{0,1\}$ such that there exist real numbers $w_1,\ldots, w_n ,t$, fulfilling
\[f(x_1,\ldots,x_n) =1 \mbox{ if and only if }  \sum_{i=1}^{n} w_i\cdot x_i \geq t,\] where $w_i$ is called  {\bf weight} of $x_i$, for $i=1,2, \ldots , n$ and $t$ is a constant called the {\bf threshold value}.
In this paper we define a new but related notion: the so called \it lattice induced threshold function \rm
and we investigate its properties.

\subsection{Historical background}

For the historical as well as  for the basic mathematical background we recommend the books \cite{CH,M,S}.
Although the main results about  classical threshold function are not algebraic, we can list some algebraic papers from several areas of algebra. In \cite{ABGG} the authors reveal connection between classical threshold functions and fundamental ideals of group-rings.
Paper \cite{H1} determines the invariance groups of threshold functions.
Paper \cite{He} proves that classical threshold functions cannot be characterized by a finite set of standard or generalized constraints.

\subsection{Motivation}

Isotone Boolean functions constitute clone; threshold fun\-ctions are not closed under  superposition, see Theorem 9.2 in \cite{CH}, so they do not constitute clone.
It is easy to see that threshold functions with positive weights and a threshold value are isotone. However, an isotone Boolean function is not necessarily threshold, e.g. $f=x\cdot y \vee w\cdot z$ is isotone, but not a threshold function. Therefore, our aim was to generalize threshold functions in the framework of lattice valued functions, in order to obtain a characterization of all isotone Boolean functions and to represent them by particular linear combinations.

\subsection{Outline}
In  Preliminaries, we present properties of closure operations, lattice valued functions, Boolean functions and classical threshold functions. Part \ref{canr} of this section contains some new structure properties of induced closure systems, including a particular canonical representation of lattice valued functions.

 In Section 3 we define threshold functions induced by complete lattices; these are Boolean functions determined by particular linear combination of lattice elements. We prove that every isotone Boolean function is a lattice induced threshold function and vice versa.

 In Section \ref{acut} we represent lattice valued up-sets on a finite Boolean lattice $B$ in the framework of cuts and lattice induced threshold functions. We construct a special lattice valued function $\overline{\beta}$ on $B$, whose co-domain is a free distributive lattice. We prove that the cuts of any lattice valued function on $B$ are contained in the collection of cuts of $\overline{\beta}$, and that all up-sets of $B$ are cuts of $\overline{\beta}$.

Section \ref{linco} is devoted to the representation of lattice valued up-sets on $B$ by linear combination of elements of the co-domain lattice. In terms of closure systems we present necessary and sufficient conditions for such representation, as well as conditions for its existence.

\section{Preliminaries}\label{prel}
\subsection{Order, lattices, closures}
Our basic notion is an ordered set, \textbf{poset}, denoted by $(P,\leq )$, where $\leq$ is an order on a set $P$. If a poset is a \textbf{lattice}, then it is denoted by $(L,\leq )$, with the meet  and the join    of $a,b\in L$ being $a\wedge b$  and $a\vee b$ respectively. The bottom and the top of a poset, if they exist, are 0 and 1, respectively. If $(P,\leq )$ is a poset, then we denote by $\bigwedge M$ and $\bigvee M$  the meet and the join of $M\subseteq P$ respectively, if they exist. We  deal with \textbf{complete lattices} and \textbf{free distributive lattices with $n$ generators}. We also use finite \textbf{Boolean lattices}, represented by all $n$-tuples of 0 and 1, ordered componentwise, and denoted by $(\{ 0,1\}^n,\leq )$.

A \textbf{closure system} on $A$ is a collection of subsets of $A$, closed under set intersection and containing $A$.

A mapping $X\mapsto \overline{X}$ on a power set $\mathcal{P}(A)$ of a set $A$, is a \textbf{closure operator} on $A$, if it fulfills: $X\subseteq \overline{X}$, $\overline{\overline{X}}=\overline{X}$, and $X\subseteq Y \Longrightarrow \overline{X}\subseteq \overline{Y}$. A subset $X$ of $A$ fulfilling $\overline{X}=X$ is said to be \textbf{closed} under the corresponding closure operator. We use the following known properties of closure systems and closure operators.

{\sl A closure system is a complete lattice under the set-inclu\-sion.}

{\sl The collection of all closed sets under the corresponding closure operator is a closure system.}

{\sl If $\mathcal{F}$ is a closure system over $A$, then the mapping \[X\mapsto\bigcap\{ Y\in \mathcal{F}\mid X\subseteq Y\}\] is a closure operator.}

The proof of the following technical lemma is straightforward.

\begin{lemma} \label{malalema}\sl
Let $\mathcal{F}$ be a closure system consisting of some up-sets on a poset $(P,\leq )$.
For  $x\in P$,  denote
\begin{equation}\overline{x}=\bigcap\{f\in{\mathcal F}\mid x\in f\}.\label{repkan}\end{equation}
Then, for all $x,y,z\in P$, the following is true:

$ a)$ $x\leq y$ implies $\overline{y}\subseteq\overline{x}$.

$ b)$ $x\in \overline{x}$.

$ c)$ $\uparrow$$x\subseteq \overline{x}$.

$ d)$ If $z\in \overline{x}$ then $\overline{z}\subseteq \overline{x}$.
\end{lemma}

More details about posets, closures and lattices can be found in e.g., \cite{DP}.

\subsection{Lattice valued functions}\label{prfuz}
Let $B$ be a nonempty set and $L$ a complete lattice. Every mapping $\mu:B\rightarrow L$ is called a \textbf{lattice valued} (\textbf{$L$-valued}) \textbf{function} on $B$.

%\begin{rem}\rm What we use here are notions and techniques from the fuzzy set theory. However, instead of  \textsl{fuzzy set}, or \textsl{lattice valued set}, we use the term \textsl{lattice valued function}, relying to the mathematical nature of fuzzy objects. The reason is that our investigation here is focused more to functions and their properties, than  to fuzzy logic intuition by which these are called fuzzy sets. However, techniques from fuzzy algebra are essential in our research. \hspace*{\fill}$\Box$ \end{rem}

Let $p\in L$. A \textbf{cut set} of an $L$-valued function $\mu:B\rightarrow L$  (a $p$-cut) is a subset $\mu_p\subseteq B$ defined by:
\[ x\in \mu_p \mbox{ if and only if } \mu(x)\geq p.\] In other words, a $p$-cut of $\mu:B\rightarrow L$ is the inverse image of the principal filter ${\uparrow}p$, generated by $p\in L$: \begin{equation}\mu_p=\mu^{-1}({\uparrow}p).\label{invim}\end{equation}

It is obvious that for $p,q\in L$,
\[\mbox{from $p\leq q$ it follows that $\mu_q\subseteq\mu_p$.}\]

The collection   $\mu_L=\{f\subseteq B\mid f=\mu_p,\mbox{ for some } p\in L\}$ of all cuts of $\mu:B\rightarrow L$ is usually ordered by set-inclusion. The following is known.

\begin{lemma}\label{intr}\sl If $\mu:B\rightarrow L$ is an $L$-valued function on $B$, then the collection $\mu_L$ of all cuts of $\mu$  is a closure system on $B$ under the set-inclusion. \end{lemma}

The following is a kind of a converse.
\begin{prop}\label{synth}\sl Let $\mathcal{F}$ be a closure system over a set $B$. Then there is a lattice $L$ and an $L$-valued function $\mu:B\rightarrow L$, such that the collection $\mu_L$ of cuts of $\mu$ is $\mathcal{F}$. \end{prop}
\begin{proof} It is straightforward to check that a required  lattice $L$ is the collection $\mathcal{F}$ ordered dually to the set-inclusion, and that $\mu:B\rightarrow L$ can be defined as $\overline{x}$ in (\ref{repkan}):
\begin{equation}\mu (x) =\bigcap\{f\in{\mathcal F}\mid x\in f\}.\label{fsynth}\end{equation}
\end{proof}

\begin{rem}\rm  From the above proof, using the notation therein, one can straightforwardly deduce that  for every $f\in\mathcal{F} $, the cut $\mu_f$ coincides with $f$, i.e., that $\mu_f=f$.
\end{rem}

\subsection{Boolean functions; threshold function}\label{thf}

A \textbf{Boolean function} is a mapping $f: \{0,1\}^n\rightarrow \{0,1\}$, $n\in \mathbb{N}$.

The domain $\{0,1\}^n$ of a Boolean function is usually ordered componentwise, with respect to the natural order $0\leq 1$
:
$(x_1,x_2,\ldots,x_n)\leq (y_1,y_2,\ldots,y_n)$ if and only if for every $i\in\{1,\ldots,n\}$, $x_i\leq y_i$.

As it is known, the poset  $(\{ 0,1\}^n,\,\leq )$ is a Boolean lattice. Moreover, every finite Boolean lattice with $n$ atoms is isomorphic to this one.

Recall that a subset $F$ of a poset $(P,\,\leq\,)$ is an \textbf{up-set} (\textbf{order semi-filter, semi-filter}) on $P$ if for all $x\in F$ and $y\in P$,
\[x\leq y \;\mbox{ implies }\; y\in F.\]

A Boolean function $f:\{0,1\}^n\rightarrow \{0,1\}$ is \textbf{isotone} (order preserving, positive, as in \cite{CH}), if for every $x,y\in \{0,1\}^n$,
from $x\leq y$, it follows that $f(x)\leq f(y)$.

The following is easy to check.
\begin{lemma}\sl\label{prva} The set $F\subseteq \{0,1\}^n$ is an order semi-filter on $(\{ 0,1\}^n,\,\leq )$  if and only if  Boolean function $f$ defined by
\[f(x)=1\;\mbox{ if and only if }\; x\in F\] is isotone. \end{lemma}

A \textbf{threshold function} is a Boolean function $f: \{0,1\}^n\rightarrow \{0,1\}$ such that there exist real numbers $w_1,\ldots, w_n ,t$, fulfilling
\[f(x_1,\ldots,x_n) =1 \mbox{ if and only if }  \sum_{i=1}^{n} w_i\cdot x_i \geq t,\] where $w_i$ is called  {\bf weight} of $x_i$, for $i=1,2, \ldots , n$ and $t$ is a constant called the {\bf threshold value}.

\subsection{Lattice valued Boolean functions}

A function $f:\{0,1\}^{n}\rightarrow L$, where $L$ is a
complete lattice, is called a \textbf{lattice valued} (\textbf{$L$-valued}) \textbf{Boolean function}.

For $f:\{0,1\}^{n}\rightarrow L$ and $p\in L$, a cut set (cut) $f_p$ is a subset of $\{0,1\}^n$ defined as above:
\[f_p=\{ x\in\{ 0,1\}^n\mid f(x)\geq p\}.\]

%A characteristic function of the cut $f_p$ is an ordinary Boolean %function
%and we denote it by the same symbol $f_p$.

 An $L$-valued Boolean function $\mu:B\rightarrow L$  is called a \textbf{lattice valued} {(\textbf{$L$-valued}) \textbf{up-set}, or a \textbf{lattice valued} (\textbf{$L$-valued}) \textbf{semi-filter} on $B$ if
from $x\leq y$ it follows that $\mu(x)\leq \mu(y)$. This definition originates from the research of lattice valued structures whose domains are ordered structures and lattices (\cite{JMST1}). The following lemma is adapted version of a result from  \cite{JMST1}.
\begin{lemma}\label{carup}\sl
Let $B$ be a Boolean lattice and $\mu:B\rightarrow L$   an $L$-valued Boolean function. Then $\mu$ is an $L$-valued up-set on $B$ if and only if all the cuts of $\mu$  are up-sets on $B$. \end{lemma}

\subsection{Canonical representation of lattice valued functions}\label{canr}

We shall use the following properties of the collection of cuts (for proofs that are not presented, see \cite{BSAT1}). In this section $B$ is a nonempty set, not necessarily equipped with operations or relations.

 Let $\mu:B\rightarrow L$ be an $L$-valued Boolean function and $(\mu_L, \leq)$ the poset with  $\mu_L=\{\mu_p\mid p\in L\}$ (the collection of cuts of $\mu$) and the order $\,\leq\,$ is the inverse  of the set-inclusion: for $\mu_p,\mu_q\in\mu_L$,
  \[\mu_p\leq\mu_q\;\mbox{ if and only if }\;\mu_q\subseteq\mu_p.\]
  \begin{lemma}\sl  $(\mu_L, \leq)$ is a complete lattice and for every collection $\{\mu_p\mid p\in L_1\}$, $L_1\subseteq L$ of cuts of $\mu$, we have
  \begin{equation}\bigcap\{\mu_p\mid p\in L_1\}=\mu_{{\vee} (p\mid p\in L_1)}.\label{pres}\end{equation}\end{lemma}

  Given an $L$-valued Boolean function $\mu :B\rightarrow L$, we define a relation $\approx$ on $L$: for $p,q\in L$
  \begin{equation} p\approx q\; \mbox{ if and only if }\; \mu_p=\mu_q.\label{tilda}\end{equation}
  The following is straightforward by (\ref{invim}).
  \begin{lemma}\sl The relation $\approx$ is an equivalence on $L$, and
  \begin{equation} p\approx q\; \mbox{ if and only if }\; {\uparrow}p\cap\mu (B)= {\uparrow}q\cap\mu (B),\label{tilda1}\end{equation}
  where $\mu (B)=\{ r\in L\mid r=\mu (x)\,\mbox{ for some } \,x\in B\}$. \end{lemma}

  We denote by $L/{\approx}$ the collection of equivalence classes under $\approx$. By (\ref{pres}), supremum of each $\approx$-class belongs to the class, as its greatest element:
  \begin{equation}\bigvee [p]_{\approx}\in [p]_{\approx}. \label{suprk}\end{equation}
  In particular, we have that for every $x\in B$
  \begin{equation} \mu (x)=\bigvee  [\mu (x)]_{\approx}.\label{slik}\end{equation}

  \begin{lemma}\sl
  The mapping $p\mapsto\bigvee  [p]_{\approx}$ is a closure operator on $L$. \end{lemma}
   The quotient $L/{\approx}$ can be ordered by the relation $\leq_{L/{\approx}}$ defined as follows:
  \[[p]_{\approx}\leq_{L/{\approx}} [q]_{\approx}\;\mbox{ if and only if } \; {\uparrow}q\cap\mu (B)\subseteq {\uparrow}p\cap\mu (B).\]
  The order $\leq_{L/{\approx}}$ of classes in $L/{\approx}$ corresponds to the order of suprema of classes in $L$ (we denote the order in $L$ by $\leq_L$):
  \begin{prop}\label{appcl}\sl The poset $(L/{\approx},\leq_{L/{\approx}})$ is a complete lattice fulfilling:

  $(i)\;$ $[p]_{\approx}\leq_{L/{\approx}} [q]_{\approx}\;\mbox{ if and only if } \; \bigvee [p]_{\approx}\leq_L \bigvee [q]_{\approx}.$

  $(ii)\;$ The mapping $[p]_{\approx}\mapsto\bigvee [p]_{\approx}$ is an injection of $L/{\approx}$ into $L$.
\end{prop}
\begin{cor}\sl The sub-poset $(\bigvee [p]_{\approx},\leq_L)$ of $L$ is a lattice isomorphic under $\bigvee [p]_{\approx}\mapsto [p]_{\approx}$ with the lattice $(L/{\approx},\leq_{L/{\approx}})$.
\end{cor}

Next we connect the lattice $(L/{\approx},\leq_{L/{\approx}})$ (Proposition \ref{appcl}) and the lattice $(\mu_L,\leq )$ of cuts of $\mu$; recall that the latter is ordered dually to inclusion.

\begin{prop}\label{obem}\sl Let $\mu:B\rightarrow L$ be an $L$-valued function on $B$. Its lattice of cuts $(\mu_L, \leq)$ is isomorphic with the lattice  $(L/{\approx},\leq_{L/{\approx}})$ of $\approx$-classes in $L$ under the mapping $\mu_p\mapsto [p]_{\approx}$. \end{prop}

  We introduce the mapping $\widehat{\mu} :B\rightarrow \mu_L$ by the construction (\ref{fsynth}), i.e., by
\begin{equation} \widehat{\mu} (x):=\bigcap\{ \mu_p\in\mu_L\mid x\in\mu_p\}. \label{can}\end{equation}

As in \cite{BSAT1}, we say that the lattice valued function $\widehat{\mu}$ is the \textbf{canonical representation} of $\mu$. Observe that its cuts are of the form $\widehat{\mu}_f$, $f=\mu_f\in\mu_L$; hence, its collection of cuts is \[\widehat{\mu}_{\mu_L}=\{ F\subseteq B\mid F=\widehat{\mu}_f, \mbox{ for some  }f\in\mu_L\}.\]
\begin{prop}\label{izok} \sl  The lattices of cuts of a lattice valued function $\mu$ and of its canonical representation $\widehat{\mu}$ coincide.
\end{prop}
\begin{proof} Indeed, for every cut $\mu_p$ of $\mu$, for the corresponding cut $\widehat{\mu}_{\mu_p}$ of $\widehat{\mu}$ we have that  $\widehat{\mu}_{\mu_p}=\mu_p$.
\end{proof}

Observe that the lattice valued function $\mu$ defined by (\ref{fsynth}) in the proof of Proposition \ref{synth}, coincides with its canonical representation, i.e., in this case we have  $\mu=\widehat{\mu}$.

Using the above properties of cuts and of the corresponding lattices, we finally give a consequence which is relevant to our present investigation.

\begin{prop}\label{homom}\sl If $\mu:B\rightarrow L$ is an $L$-valued function on $B$ and
$\mu (a)=\mu (b)\vee\mu (c)$
for some $a,b,c\in B$, then also for the canonical representation $\widehat{\mu}$ of $\mu$ analogously holds
$\;\widehat{\mu}(a)=\widehat{\mu}(b)\vee\widehat{\mu}(c).
$
\end{prop}
\begin{proof} By Proposition \ref{izok}, the lattices of cuts $(\mu_L,\leq )$ and $(\widehat{\mu}_{\mu_L},\leq )$ of $\mu$ and of its canonical representation $\widehat{\mu}$ respectively, coincide. The co-domain lattice of $\widehat{\mu}$ is moreover $(\mu_L,\leq )$. By Proposition \ref{obem}, $(\mu_L,\leq )$ is isomorphic with the lattice $(L/{\approx},\leq_{L/{\approx}})$ of $\approx$-classes in $L$, and the latter is isomorphic with the lattice $(\bigvee [p]_{\approx},\leq_L)$ of maximal elements of $\approx$-classes in $L$. But the last one is a sub-poset in $L$. Now, for every $x\in B$, $\mu (x)$ is by (\ref{slik}) an element in the lattice $(\bigvee [p]_{\approx},\leq_L)$ which is a sub-poset in $L$. Therefore, if the equality $\mu (a)=\mu (b)\vee\mu (c)$ holds in $L$, then it also holds in $(\bigvee [p]_{\approx},\leq_L)$. Due to the mentioned lattice isomorphisms we have the analogue equality also for $\widehat{\mu}$.
\end{proof}
We recall that the equality $\widehat{\mu}(a)=\widehat{\mu}(b)\vee\widehat{\mu}(c)$  equivalently can be presented as
$\widehat{\mu}(a)=\widehat{\mu}(b)\cap\widehat{\mu}(c),$ since the order in $(\mu_L,\leq )$ is dual to set-inclusion; therefore the join in this lattice is actually the set intersection.

\begin{rem}\rm
The opposite implication to the one in Proposition \ref{homom} does not hold in general. Indeed, let $B=\{a,b,c,d\}$, and let  $L$ be the lattice given in Figure 1.
\begin{center}
% This is a LaTeX picture output by TeXCAD.
% File name: [12fig.pic].
% Version of TeXCAD: 4.3
% Reference / build: 30-Jun-2012 (rev. 105)
% For new versions, check: http://texcad.sf.net/
% Options on the following lines.
%\grade{\on}
%\emlines{\off}
%\epic{\off}
%\beziermacro{\on}
%\reduce{\on}
%\snapping{\on}
%\pvinsert{% Your \input, \def, etc. here}
%\quality{8.000}
%\graddiff{0.005}
%\snapasp{1}
%\zoom{11.3137}
\unitlength 1mm % = 2.845pt
\linethickness{0.4pt}
\ifx\plotpoint\undefined\newsavebox{\plotpoint}\fi % GNUPLOT compatibility
\begin{picture}(92,40)(0,-5)
\put(18,32){\circle{2}}
\put(18,23){\circle{2}}
\put(78,25){\circle{2}}
\put(10,15){\circle{2}}
\put(70,17){\circle{2}}
\put(17,22){\line(-1,-1){6}}
\put(77,24){\line(-1,-1){6}}
\put(19,22){\line(1,-1){6}}
\put(79,24){\line(1,-1){6}}
\put(26,15){\circle{2}}
\put(86,17){\circle{2}}
\put(11,14){\line(1,-1){6}}
\put(71,16){\line(1,-1){6}}
\put(25,14){\line(-1,-1){6}}
\put(85,16){\line(-1,-1){6}}
\put(18,31){\line(0,-1){7}}
\put(18,7){\circle{2}}
\put(78,9){\circle{2}}
\put(18,4){\makebox(0,0)[cc]{$s$}}
\put(7,15){\makebox(0,0)[cc]{$p$}}
\put(29,15){\makebox(0,0)[cc]{$q$}}
\put(21,25){\makebox(0,0)[cc]{$r$}}
\put(18,35){\makebox(0,0)[cc]{1}}
\put(7,6){\makebox(0,0)[cc]{$(L,\leq )$}}
\put(18,-2){\makebox(0,0)[cc]{Figure 1}}
\put(78,-2){\makebox(0,0)[cc]{Figure 2}}
\put(78,6){\makebox(0,0)[cc]{$\{ a,b,c,d\}$}}
\put(64,17){\makebox(0,0)[cc]{$\{ a,b\}$}}
\put(93,17){\makebox(0,0)[cc]{$\{a,c\}$}}
\put(78,29){\makebox(0,0)[cc]{$\{ a\}$}}
\put(92,9){\makebox(0,0)[cc]{$(\mathcal{F},\supseteq )$}}
\put(61,34){\line(0,-1){1}}
\put(61,33){\line(0,1){0}}
\end{picture}
\end{center}

We define an $L$-valued function $\mu:B\rightarrow L$ as follows:
\[\mu =\left(\begin{array}{cccc}a&b&c&d\\1&p&q&s\end {array}\right).\]
The cuts of $\mu$ are:

$\mu_L = \{\mu_1=\mu_r= \{a\}, \mu_p=\{a,b\}, \mu_q=\{a,c\}$,  $\mu_s=\{a,b,c,d\}\}$.

The lattice $(\mu_L,\supseteq)$ is depicted in Figure 2. The canonical representation of this lattice valued function is $\widehat{\mu}:B\rightarrow \mu_L$ and it is given by
\[\widehat{\mu} =\left(\begin{array}{cccc}a&b&c&d\\\{a\}&\{a,b\}&\{a,c\}&\{a,b,c,d\}\end {array}\right).\]
Now, observe that $\;\widehat{\mu}(a)=\widehat{\mu}(b)\vee\widehat{\mu}(c)$. However, it is not true that $\mu(a)=\mu(b)\vee\mu(c)$. \hfill$\Box$\end{rem}

\section{Threshold functions induced by complete lattices}

In this part we introduce threshold functions induced by complete lattices, and we use them for investigation  of isotone Boolean functions and for their representation.

We deal with functions over the Boolean lattice $(\{ 0,1\}^n,\leq)$, and we use the complete lattice $L$ in which the bottom and the top are (also) denoted by 0 and 1 respectively; however, it is  clear from the context whether 0 (1) is a component in some $(x_1,\ldots ,x_n)\in\{ 0,1\}^n$, or it is from $L$.\\

For $x\in\{0,1\}$, and $w\in L$, we define a  mapping $L\times\{ 0,1\}$ into $L$ denoted by "$\cdot$", as follows:
\begin{equation}w\cdot x := \left\{ \begin{array}{lll} w,& \mbox{ if } & x=1\\ 0, & \mbox{ if }& x=0. \end{array}\right. \label{skmn}
\end{equation}

A function $f:\{0,1\}^n\rightarrow \{0,1\}$ is a \textbf{lattice induced threshold function}, if there is a complete lattice $L$ and
$w_1,\ldots, w_n,t\in L$, such that

\begin{equation}f(x_1,\ldots,x_n) =1 \mbox{ if and only if }  \bigvee_{i=1}^{n} (w_i\cdot x_i) \geq t.\label{litf}\end{equation}

\begin{prop}\sl
Every lattice induced threshold function is isotone.
\end{prop}
\begin{proof}
Let $L$ be a complete lattice and $w_1,\ldots, w_n,t\in L$, and $f:\{0,1\}^n\rightarrow \{0,1\}$ a lattice induced threshold function.

Let $(x_1,x_2,\ldots,x_n)\leq (y_1,y_2,\ldots,y_n)$. Then,
for every $i$, we have  $w_i\cdot x_i\leq w_i\cdot y_i$, by the definition (\ref{skmn}).
Hence,
\[ \bigvee_{i=1}^{n} (w_i\cdot x_i ) \leq \bigvee_{i=1}^{n} (w_i\cdot y_i).\]

Therefore, if  $f(x_1,\ldots,x_n) =1$, then \[ \bigvee_{i=1}^{n} (w_i\cdot x_i)\geq t,\; \mbox{  and hence }\;
 \bigvee_{i=1}^{n} (w_i\cdot y_i)\geq t.\]

 This implies $f(y_1,\ldots,y_n) =1$ and we obtain
\[f(x_1,\ldots,x_n)\leq f(y_1,\ldots,y_n),\] which proves that $f$ is an isotone function.
\end{proof}

\begin{thm}\label{thm1n}\sl
Every isotone Boolean function is a lattice induced threshold function.
\end{thm}

\begin{proof}
We  prove that for every $n\in \mathbb{N}$, there is a lattice $L$ such that every isotone Boolean function is a lattice induced threshold
function over $L$.

Let $n\in \mathbb{N}$. We take $L$ to be a free distributive lattice with $n$ generators $w_1,$ $w_2,$ $\ldots ,$ $w_n$ (the join and meet of empty set of generators are also counted here, as the bottom and the top of $L$, respectively).
Recall that every element in a free distributive lattice can be uniquely represented in a   "conjunctive normal form" by means of generators
(i.e., every element is a meet of elements of the type $\bigvee_{i\in J} w_j$, where $J\subseteq \{1,\ldots,n\}$,) see e.g., \cite{crawley}.

Therefore, for $x,y\in L$, if $x=\bigwedge_{k=1}^{p} \bigvee_{j\in I_k}w_j$ and $y= \bigwedge_{k=1}^{l} \bigvee_{s\in J_k}w_s$,
then $x\leq y$ if and only if for every $u\in\{1,\ldots, l\}$ there is $k\in\{1,\ldots, p\}$ such that $I_k\subseteq J_u$. (*)

%%%%%%%%%%%%%%%%%%%%%%
%%%%%%%%%%%%%%%%%%%%%%%%

Let $f:\{0,1\}^n\rightarrow \{0,1\}$ be an isotone Boolean function. Let $F$ be the corresponding order semi-filter on $\{ 0,1\}^n$ (according to Lemma \ref{prva}).
Further, let $m_1$,\ldots ,$m_p$ be minimal elements of this semi-filter. Let $I_1$,\ldots,$I_p$ be subsets of $\{1,2,\ldots,n\}$, i.e.,
sets of indices, such that $i\in I_k$ if and only if $i-th$ coordinate of $m_k$
 is equal to 1.

For the threshold $t\in L$ associated to the given function $f$ we take
\[t=\bigwedge_{k=1}^{p} \bigvee_{j\in I_k}w_j.\]

Now, we prove that
 \begin{equation}f(x_1,\ldots,x_n)=1\;\mbox{ if and only if }\;\bigvee_{i=1}^{n} (w_i\cdot x_i) \geq t.\label{thm1}\end{equation}
  Indeed, from $f(x_1,\ldots,x_n)=1$, it follows that there is a minimal element $m_l$ in the corresponding semi-filter, such that
   $(x_1,\ldots,x_n)\geq m_l$. Hence, \[\bigvee_{i=1}^{n}( w_i\cdot x_i )\geq \bigvee_{j\in I_l}w_j \geq \bigwedge_{k=1}^{p} \bigvee_{j\in I_k}w_j=t.\]
  Now we suppose that \[\bigvee_{i=1}^{n} (w_i\cdot x_i )\geq \bigwedge_{k=1}^{p} \bigvee_{j\in I_k}w_j\] for an ordered n-tuple $(x_1,\ldots,x_n)$.
  Let $I\subseteq \{1,\ldots n\}$ be the set of indices such that $x_i=1$ if and only if $i\in I$. We prove that there is $s\in\{1,\ldots p\}$ such that
   $I_s\subseteq I$. This follows directly from the above mentioned (*) property  of the free distributive lattice with $n$ generators $w_1,\ldots, w_n$.

  Now, it follows that $(x_1,\ldots, x_n)\geq (y_1,\ldots,y_n),$ where $y_i=1$ if and only if $i\in I_s$.
  Therefore, \[\bigvee_{i=1}^{n} (w_i\cdot y_i) \geq t,\;\mbox{ and }\; f(y_1,\ldots,y_n)=1.\] This finally implies  $\;f(x_1,\ldots,x_n)=1$.
\end{proof}

\begin{rem}\rm
In the previous proposition it is proved not only that every isotone $n$-ary Boolean function is a lattice induced threshold function, but also that the corresponding lattice in each case can be the free distributive lattice with $n$ generators. \hfill$\Box$\end{rem}

\section{Representation of lattice valued up-sets by cuts}\label{acut}

To the end of the article, an arbitrary finite Boolean lattice is denoted by $B$ and  it is represented by $(\{ 0,1\}^n,\leq )$.

 In this part, we represent lattice valued up-sets on $B$  in the framework of cut sets and lattice induced threshold functions.

Our main tool here is a particular lattice valued Boolean function, defined as follows.

Let $B=(\{ 0,1\}^n,\leq )$, $n\in \mathbb{N}$,  $L_D$ a free distributive lattice with $n$ generators $w_1,\ldots ,w_n$ and $\overline{\beta} :B\rightarrow L_D$, an $L_D$-valued function on $B$ defined in the following way: for $x=(x_1,\ldots ,x_n)\in B$
\begin{equation}\overline{\beta} (x)=\bigvee_{i=1}^n (w_i\cdot x_i),\label{can2}\end{equation}
where the function "$\,\cdot\,$" is defined by (\ref{skmn}). By the definition, $\overline{\beta}$ is uniquely (up to a permutation of generators $w_i$) determined by a finite Boolean lattice $B=(\{ 0,1\}^n,\leq )$, i.e., by a positive integer $n$.

What we prove next is that $\overline{\beta}$ can be considered as the main representative of all lattice valued up-sets on $B$, with respect to collections of cuts.

\begin{thm}\label{main}\sl
 Every up-set of a finite Boolean lattice $B=(\{ 0,1\}^n,\leq )$, $n\in \mathbb{N}$,
is a cut of  $\overline{\beta}$.
\end{thm}
\begin{proof}
Let $B=(\{ 0,1\}^n,\leq )$ be a finite Boolean lattice.   By  Theorem \ref{thm1n},  every isotone Boolean function with $n$ variables is a lattice induced threshold
function over a free distributive lattice $L_D$ with $n$ generators.
By Lemma \ref{prva}, every up-set on $B$ is (as a characteristic function) an isotone Boolean function $f: \{0,1\}^n\rightarrow \{ 0,1\}$.  This means that there are elements $w_1,\ldots ,w_n\in L_D$ (generators), such that
 \[f(x_1,\ldots,x_n) =1 \mbox{ if and only if }  \bigvee_{i=1}^{n} (w_i\cdot x_i) \geq t,\;\mbox{  for some }\;t\in L_D.\]
Thereby, due to (\ref{can2}),
\[f(x_1,\ldots,x_n) =1 \mbox{ if and only if } \overline{\beta}(x)\geq t\;\mbox{ if and only if }x\in \overline{\beta}_t.\]
Hence the up-set whose characteristic function is $f$, coincides with the cut $\overline{\beta}_t$ of $\overline{\beta}$.
\end{proof}

Due to Lemma \ref{carup}, we have the following obvious consequence of Theorem \ref{main}.
\begin{cor}\sl
The $L_D$-valued function $\overline{\beta}$ defined by $(\ref{can2})$ is an $L_D$-valued up-set on $B$. Moreover, the collection of cuts of every $L$-valued up-set on $B$ (for any $L$) is contained in the collection of cuts of $\overline{\beta}$.
\end{cor}

\section{Linear combinations}\label{linco}

In the previous section, a technique for representing lattice valued up-sets on a finite Boolean algebra by cuts of particular lattice valued functions is presented. Next we give a special name to the expression by which these functions are formulated.

 Let $B=(\{ 0,1\}^n,\leq )$ be a Boolean lattice, $L$ a complete lattice, $x=(x_1,\ldots ,x_n)\in B$ and $w_1,\ldots ,w_n\in L$. Further, let  the binary function "$\cdot$" which maps $L\times\{ 0,1\}$ into $L$ be defined by (\ref{skmn}).
 Then the term
 \begin{equation}\bigvee_{i=1}^{n} (w_i\cdot x_i),\label{link}\end{equation}
 is a \textbf{linear combination} of elements $w_1,\ldots,w_n$ from $L$. Obviously, if $x=(x_1,\ldots ,x_n)$ is considered to be a variable over $B$, then the expression (\ref{link}) determines an $L$-valued function on $B$.

Let us mention that in the previous section, where we use the above expression (formula (\ref{can2})), the lattice is the free distributive one, $L_D$, and $w_i$ are generators of $L_D$. In the present definition, (\ref{link}), $w_i$ are arbitrary elements of the lattice $L$.

Observe also that in the case of formula (\ref{can2}), the corresponding $L_D$-valued function is  $\overline{\beta}$ and the following is obviously true: \textsl{the closure system consisting  of all up-sets on $B$ is a collection of cuts of $\overline{\beta}$.}

    Next we analyze the analogue problem taking an arbitrary lattice $L$ instead of $L_D$. Namely, we start with a closure system  $\mathcal{F}$ consisting of some up-sets on $B=(\{ 0,1\}^n,\leq )$, and we try to find a lattice $L$ and $w_1,\ldots ,w_n\in L$, such that the family of cuts of the function $(\ref{link})$ over this lattice (a linear combination of elements from $L$) coincides with $\mathcal{F}$.

  The answer to the above problem is not generally positive, as shown by the following example.

  First we advance a property of cuts, which is needed in the example.
  \begin{lemma}\label{lem6}\sl
Let $\mu:B\rightarrow L$ be a lattice valued up-set, such that its collection of cuts is $\mathcal F$. If $ {\uparrow}a\in{\mathcal F}$ and $\mu(a)=p$, then $\mu_p={\uparrow}a$.
\end{lemma}
\begin{proof}
Suppose that ${\uparrow}a=\mu_q$, for $q\in L$. By $\mu(a)=p$, and $a\in\mu_q$, we have that $p\geq q$ and thus $\mu_p\subseteq\mu_q$. On the other hand, suppose that $b\in\mu_q$, i.e., that $b\geq a$. Then, $\mu(b)\geq \mu(a)=p$ and $b\in\mu_p$. Hence, $\mu_p=\mu_q$.
\end{proof}

\begin{example}\rm \label{ex2}

Let $B=(\{ 0,1\}^2,\leq )$ be the four element Boolean lattice and
\[\mathcal{F}=\{ \{(1,1)\},\{(1,1), (1,0)\},\{(1,1),(1,0),(0,1)\},\{(1,1),(1,0),(0,1),(0,0)\}\}\] a closure system consisting of some up-sets on $B$.

We show that there is no lattice $L$, hence neither there is an $L$-valued function $\nu:B\rightarrow L$, such that there are $w_1,w_2\in L$ fulfilling that for all $x_1,x_2\in\{ 0,1\}$
\[\nu(x_1,x_2) = (w_1\cdot x_1) \vee (w_2\cdot x_2)\] and that the collections of cuts of $\nu$ is $\mathcal{F}$.

Indeed, suppose that there is a lattice $L$ and elements $w_1, w_2\in L$, such that
$\nu(x_1,x_2) = (w_1\cdot x_1) \vee (w_2\cdot x_2)$, for all $x_1,x_2\in\{ 0,1\}$.

Then, $\nu(0,0)=0\in L_1$, $\nu(0,1)=w_2$, $\nu(1,0)=w_1$ and $\nu(1,1)=w_1\vee w_2$. Now, since the cuts of $\nu$ are supposed to be elements from $\mathcal{F}$, and cuts are up-sets in $B$, we have that $\nu_{w_1\vee w_2}=\{(1,1)\}$, and $w_1\vee w_2$ would be the top element of the lattice $L$: otherwise the empty set would be a cut of this lattice valued function.   Further, by Lemma \ref{lem6}, $\nu_{w_1}=\{(1,1), (1,0)\}$, $\nu_{w_2}=\{(1,1), (1,0), (0,1)\}$ and $\nu_{0}=\{(1,1), (1,0), (0,1), (0,0)\}.$
Since $(1,0)\in \nu_{w_2}$, we have that $\nu(1,0)\geq w_2$, i.e., $w_1\geq w_2$. Hence, $w_1\vee w_2=w_1$, which contradict the assumption that $\nu_{w_1\vee w_2}\not = \nu_{w_1}$.

Hence, the up-sets from the collection $\mathcal{F}$ can not be represented as cuts of an $L$-valued function in the form (\ref{link}).
\hfill$\Box$
\end{example}

According to the above analysis, first we deal with the following problem:

\textsl{Find necessary and sufficient conditions under which a lattice valued up-set $\mu: B\rightarrow L$ on a finite Boolean lattice $B=(\{ 0,1\}^n,\leq )$ can be represented  by the linear combination \[\mu(x)=\bigvee_{i=1}^n (w_i\cdot x_i)\] over $L$ ($x=(x_1,\ldots ,x_n)\in\{ 0,1\}^n, w_1,\ldots ,w_n\in L$).
}

Starting with finite lattices $M$ and $L$ with the bottom elements $0_M$ and $0_L$ respectively, we say that a mapping $\mu:M\rightarrow L$ is a \textbf{$0$--$\vee$--homomorphism}, if for all $x,y\in M$
\begin{eqnarray*}\mu (x\vee y)&=&\mu (x)\vee \mu (y)\;\;\mbox{ and }\\\mu (0_M)&=&0_L.\end{eqnarray*}

In particular, if $\mu$ maps  a Boolean lattice $B=\{ 0,1\}^n$ into $L$, the condition that $\mu$ is a $0$--$\vee$--homomorphism from $B$ to $L$ is equivalent with the following two conditions (observe that $B$ is finite):
for every collection $A$ of some atoms in $B$ \begin{equation}(i)\;\;\mu (\bigvee A)=\bigvee\mu ( A)\;\;\mbox{ and }\;\;(ii)\;\;\mu(0,\dots,0)=0.\label{fat}\end{equation}

\begin{thm}\sl \label{thom}
Let $B=(\{ 0,1\}^n,\leq )$ be a finite Boolean lattice and $L$ an arbitrary complete lattice. Then an $L$-valued up-set $\mu: \{0,1\}^n\rightarrow L$ can be represented in the form
\[\mu(x)=\bigvee_{i=1}^n (w_i\cdot x_i)\] for some elements $w_1,\ldots ,w_n\in L$ if and only if $\mu$ as a mapping from $B$ to $L$ is a $0$--$\vee$--homomorphism.
\end{thm}

\begin{proof} Let $\mu: \{0,1\}^n\rightarrow L$  be a $0$--$\vee$--homomorphism from $B=\{ 0,1\}^n$ to $L$.
 By our notation,  $a_i$ is an atom in $B$, namely $a_i=(0,\dots, 1, \dots, 0)$, where $1$ is in the $i$-th place. Define $w_i=\mu(a_i)$. The mapping $\mu$ need not be an injection, hence some of the elements $w_i$ and $w_j$ might coincide. Now, for the bottom element in $B$, i.e.,
if $x=(0,\dots,0)$, by the assumption we have \[\mu(x)=\bigvee_{i=1}^n (w_i\cdot x_i) = 0.\]  If $J\subseteq \{1,\dots, n\}$ is a set of indices and $x=(x_1,\ldots ,x_n)\in\{0,1\}^n$ has $1$ in places that corresponds to $J$, then by $(i)$ in (\ref{fat}), \[\mu(x)=\mu(\bigvee_{i\in J} a_i)= \bigvee_{i\in J}w_i = \bigvee_{i=1}^n (w_i\cdot x_i).\]

We prove the opposite implication by contraposition. Suppose that there are atoms $\{a_i\mid i\in J\}$, for $J\subseteq \{1,\dots, n\}$ such that
\[\mu(\bigvee_{i\in J} a_i)\neq \bigvee_{i\in J}\mu( a_i).\]  From \[\mu(x)=\bigvee_{i=1}^n (w_i\cdot x_i),\] we would have that $\mu(a_i)=w_i$ for all $i\in J$, and that
 \[\mu(\bigvee_{i\in J} a_i)= \bigvee_{i\in J}w_i,\]
 which is not true by the assumption.
\end{proof}

Next we analyze the same problem in the framework of families of cuts and canonical representations of lattice valued Boolean functions (Section \ref{prfuz}).

As  in Lemma \ref{malalema}, if $\mathcal{F}$ is a closure system consisting of some up-sets on  $B=(\{ 0,1\}^n,\leq )$, then
for  $x\in P$,  we define
\begin{equation}\overline{x}=\bigcap\{f\in{\mathcal F}\mid x\in f\}.\label{repkan1}\end{equation}

\begin{prop}\sl \label{nec}
Let $\mathcal{F}$ be a closure system of some up-sets on $B$. If $\mathcal{F}$ is a family of cuts of an $L$-valued up-set $\mu$ on $B$
represented by a linear combination over $L$, then the following holds: for all $x,y\in B$  \begin{equation}\mbox{from $\overline{x}\subseteq \overline{y}$ it follows that $\overline{x\vee y}=\overline{x}$.}\label{inclu}\end{equation}
\end{prop}
\begin{proof} By assumption, a closure system $\mathcal{F}$ on $B$ consisting of some up-sets on $B$ is a collection of cuts of a fuzzy set $\mu: \{0,1\}^n\rightarrow L$, for some lattice $L$, i.e., $\mu_L=\mathcal{F}$. We also assume that $\mu$ can be represented as a linear combination over $L$.

Now, if $\widehat{\mu}:B\rightarrow \mathcal{F}$ is the canonical representation of $\mu$, then by (\ref{can}), for every $x\in B$, $\widehat{\mu}(x)=\overline{x}$.

Suppose that
there are $x,y\in B$ such that $\overline{x}\subseteq \overline{y}$ and $\overline{x\vee y}\not = \overline{x}$.
Then $x$ and $y$ are incomparable, since by Lemma \ref{malalema} $a)$, $x\leq y$ would imply that $\overline{y}\subseteq\overline{x}$ and thus
$\overline{y}=\overline{x}=\overline{x\vee y}$. Similarly, from $y\leq x$ it would follow that $x\vee y = x$ and $\overline{x\vee y}=\overline{x}$.

Therefore,
we have that $\widehat{\mu}(x)\vee\widehat{\mu}(y) = \widehat{\mu}(x)\cap\widehat{\mu}(y) = \overline{x}\cap\overline{y} = \overline{x} \not = \overline{x\vee y} =  \widehat{\mu}(x\vee y).  $
Applying Proposition \ref{homom}, by contraposition we obtain that $\mu (x)\vee\mu (y)\neq\mu (x\vee y)$. Now, by Theorem \ref{thom}, $\mu$ is not representable by a linear combination.
\end{proof}

\begin{example}\rm
Let us consider the family $\mathcal{F}$ from Example \ref{ex2}. We already proved that this family is not the collection of cuts for a lattice valued function representable by a linear combination. If we define a mapping from $B$ to ${\mathcal F}$ by (\ref{repkan1}),
then the condition (\ref{inclu}) from Proposition \ref{nec} is not satisfied. \hspace*{\fill}$\Box$
\end{example}

\begin{thm}\sl
Let $\mathcal{F}$ be a closure system of some
up-sets on a Boolean algebra $B$ and for $x\in B$, define $\overline{x}$ by $(\ref{repkan1})$: \[\overline{x}=\bigcap\{f\in{\mathcal F}\mid x\in f\}.\]
The following conditions are equivalent:

$(i)$ for all $x,y\in B$  \[\mbox{from $\overline{x}\subseteq \overline{y}$ it follows that $\overline{x\vee y}=\overline{x}$.}\]

$(ii)$ for all $x,y\in B$,  $\overline{x\vee y}=\overline{x}\cap\overline{y}$.

$(iii)$
There is a lattice $L$ such that $\mathcal{F}$ is a family of cuts of an $L$-valued up-set on $B$ which can be
represented as a linear combination over $L$.
\end{thm}
\begin{proof}
$(iii)$ $\rightarrow$ $(i)$   Proved in Proposition \ref{nec}.

$(i)$ $\rightarrow$ $(ii)$

From $\overline{x\vee y}\subseteq \overline{x}$ and  $\overline{x\vee y}\subseteq \overline{y}$, it follows that
$\overline{x\vee y}\subseteq\overline{x}\cap\overline{y}$.

Suppose that $z\in\overline{x}\cap\overline{y}$. Then, $z\in \overline{x}$ and  $z\in \overline{y}$. By Lemma \ref{malalema} $ d)$,
we have that $\overline{z}\subseteq \overline{x}$ and  $\overline{z}\subseteq \overline{y}$. Hence, by $(i)$, $\overline{x\vee z}=\overline{z}$ and
$\overline{y\vee z}=\overline{z}$. Applying $(i)$ again (and $\overline{x\vee z}\subseteq \overline{y\vee z}$), we obtain $\overline{x\vee y\vee z}=\overline{x\vee z} = \overline{z}$. By $ x\vee y\leq x\vee y\vee z$, we have that
$\overline{x\vee y\vee z}\subseteq \overline{x\vee y}$. Hence, $\overline{z}\subseteq \overline{x\vee y}$, and since by Lemma \ref{malalema} $ b)$, $z\in\overline{z}$, it follows that $ z\in\overline{x\vee y}$.

Therefore,  $\overline{x\vee y}\supseteq\overline{x}\cap\overline{y}$.

$(ii)$ $\rightarrow$ $(i)$

This implication follows directly, as a property of inclusion and set operations.

$(ii)$ $\rightarrow$ $(iii)$

We start from a Boolean algebra $B$ and from ${\mathcal F}\subseteq {\mathcal P}(B)$, which is closed under intersection and satisfying also $\overline{x\vee y}=\overline{x}\cap\overline{y}$, for  $x,y\in B$.

Then we use Proposition \ref{synth} and formula (\ref{fsynth}) in order to construct a lattice valued function  $\mu:B\rightarrow {\mathcal F}$:
 \[\mu(x)=\bigcap\{ f\in {\mathcal F}\mid x\in f\}.\]

 As mentioned in Section \ref{prfuz}, in this case $\mu$ coincides with its canonical representation $\widehat{\mu}$, that is $\mu=\widehat{\mu}$. In addition, we have that for every $x\in B$, $\mu (x)= \overline{x}$. Hence, by $(ii)$,
\[\mu (x\vee y) = \mu(x)\cap \mu(y).\]
Further, \[\mu (0) = \overline{0} = B\in {\mathcal F}.\]($B$ is the bottom element in ${\mathcal F}$, since the latter is ordered dually to set-inclusion.)

 Then by Theorem \ref{thom}, we have that $\mu: \{0,1\}^n\rightarrow {\mathcal F}$ can be represented in the form
\[\mu(x)=\bigvee_{i=1}^n (w_i\cdot x_i)\] for some elements $w_1,\ldots ,w_n\in {\mathcal F}$.

\end{proof}

Now we are ready to solve the second problem related to representation of lattice valued up-sets by linear combinations:

\textsl{Given a lattice valued up-set $\mu: B\rightarrow L$ on a finite Boolean lattice $B=\{ 0,1\}^n$, find a lattice $L_1$ and a lattice valued function $\nu: B\rightarrow L_1$ defined by the formula \[\nu(x)=\bigvee_{i=1}^n (w_i\cdot x_i)\] where $w_1,\ldots ,w_n\in L_1$, such that the collections of cuts of $\mu$ and $\nu$ coincide.
}

\begin{cor}\sl
 For a lattice valued up-set $\mu: B\rightarrow L$ on a finite Boolean lattice $B=\{ 0,1\}^n$, there is a lattice $L_1$ and a lattice valued function $\nu: B\rightarrow L_1$ defined by the formula \[\nu(x)=\bigvee_{i=1}^n (w_i\cdot x_i)\] such that the collections of cuts of $\mu$ and $\nu$ coincide if and only if   $\overline{x\vee y}=\overline{x}\cap\overline{y}$ for  $x,y\in B$, where the operator $\bar{\;}$ is defined by cuts of $\mu$.
\end{cor}

\begin{proof}
The "only if"  part follows from the $(ii)$ $\rightarrow$ $(iii)$ from the previous theorem.

Suppose that there is a lattice valued up-set $\nu: B\rightarrow L_1$ is defined by the formula \[\nu(x)=\bigvee_{i=1}^n (w_i\cdot x_i),\] having the same cuts as the starting lattice valued function $\mu$. Then, by Theorem \ref{thom},  $\nu$ is a $0$-$\vee$-homomorphism. Furthermore, by Proposition \ref{homom},
the canonical lattice valued function is also $0$-$\vee$-homomorphism, which is equivalent with condition $\overline{x\vee y}=\overline{x}\cap\overline{y}$ for  $x,y\in B$.
\end{proof}

\section{Acknowledgments}

Research of the first author is partially supported by
the NFSR of Hungary (OTKA), grant no.\ K83219. Supported by the European Union and co-funded by the European Social Fund
   under the project ``Telemedicine-focused research activities on the field of Mathematics,
   Informatics and Medical sciences'' of project number
     ``T\'AMOP-4.2.2.A-11/1/KONV-2012-0073''.

Research of the second and the third author is supported by Ministry Education and Science, Republic of Serbia, Grant No.\ 174013.

Research of all authors is partially supported by the Provincial Secretariat
for Science and Technological Development, Autonomous  Province of
Vojvodina, grant "Ordered structures and applications".

%\begin{acknow}

%\end{acknow}
%

\end{document}